\documentclass{amsart}

\usepackage{tikz}
\usepackage{xcolor}

\usepackage{latexsym,amsmath,amssymb,epic}
\usepackage{latexsym,amssymb,epic}
\usepackage{amsmath,amsthm}
\usepackage{color}

\usepackage{soul}

\usepackage{hyperref}
\hypersetup{
	colorlinks=true, 
	linktoc=all, 
	linkcolor=blue} 

\newtheorem{theorem}{Theorem}[section]
\newtheorem{lemma}[theorem]{Lemma}
\newtheorem{corollary}[theorem]{Corollary}

\newtheorem{remark}[theorem]{Remark}

\begin{document}

	\title[Congruences for $k$-Elongated Partition Diamonds]{Elementary Proofs of Infinitely Many Congruences for $k$-Elongated Partition Diamonds}

	\author{Robson da Silva}
	\address{Universidade Federal de S\~ao Paulo, S\~ao Jos\'e dos Campos, SP 12247--014, Brazil}
	\email{silva.robson@unifesp.br}

	\author{Michael D. Hirschhorn}
	\address{School of Mathematics, UNSW, Sydney, Australia 2052}
	\email{m.hirschhorn@unsw.edu.au}

	\author{James A. Sellers}
	\address{Department of Mathematics and Statistics, University of Minnesota Duluth, Duluth, MN 55812, USA}
	\email{jsellers@d.umn.edu}

	
	
	
	\maketitle
	
	
	\begin{abstract}
 In 2007, Andrews and Paule published the eleventh paper in their series on MacMahon’s partition analysis, with a particular focus on broken $k$-diamond partitions.  On the way to broken $k$-diamond partitions, Andrews and Paule introduced the idea of $k$-elongated partition diamonds.  Recently, Andrews and Paule revisited the topic of $k$--elongated partition diamonds.  Using partition analysis and the Omega operator, they proved that the generating function for the partition numbers $d_k(n)$ produced by summing the links of $k$--elongated plane partition diamonds of length $n$ is given by $\frac{(q^2;q^2)_\infty^k}{(q;q)_\infty^{3k+1}}$  for each $k\geq 1.$  A significant portion of their recent paper involves proving several congruence properties satisfied by $d_1, d_2$ and $d_3$, using modular forms as their primary proof tool. In this work, our goal is to extend some of the results proven by Andrews and Paule in their recent paper by proving infinitely many congruence properties satisfied by the functions $d_k$ for an infinite set of values of $k.$  The proof techniques employed are all elementary, relying on generating function manipulations and classical $q$-series results.

	\end{abstract}
\section{Introduction}  

In 2007, Andrews and Paule \cite{ap11} published the eleventh paper in their series on MacMahon’s partition analysis, with a particular focus on the combinatorial objects that they called broken $k$-diamond partitions.  In that paper, Andrews and Paule introduced the idea of $k$-elongated partition diamonds.  Recently, Andrews and Paule \cite{ap13} revisited the topic of $k$--elongated partition diamonds.  Using partition analysis and the Omega operator, they proved an elegant representation for the generating function for the partition numbers $d_k(n)$ produced by summing the links of $k$--elongated plane partition diamonds of length $n.$ They then proceeded to prove several congruence properties satisfied by $d_1, d_2$ and $d_3$ using modular forms as their primary proof tool.  

More recently, Smoot \cite{Sm} extended the congruence work of Andrews and Paule, refining a conjectured infinite family of congruences that appears in \cite{ap13} and proving the following via modular forms: 

\begin{theorem}
\label{SmootCongs}
For all $\alpha\geq 0$ and for all $n\geq 0$ such that $8n\equiv 1\pmod{3^\alpha},$
$$
d_2(n) \equiv 0 \pmod{3^{2\lfloor \alpha/2 \rfloor +1}}.
$$
\end{theorem}

In this work, our goal is to extend some of the results proven by Andrews and Paule in \cite{ap13} by proving infinitely many congruence properties satisfied by the functions $d_k$ for an infinite set of values of $k.$  The proof techniques employed below are all elementary, relying on generating function manipulations and classical $q$-series results.

As noted in Andrews and Paule \cite{ap13}, we know that the generating functions in question are of the form
\begin{equation}
\sum_{n=0}^{\infty} d_k(n)q^n = \frac{f_2^k}{f_1^{3k+1}}
\label{gf}
\end{equation}
where $f_r = (1-q^r)(1-q^{2r})(1-q^{3r})(1-q^{4r})\dots$ is the usual $q$--Pochhammer symbol.

In order to provide the elementary proofs that we desire, we will require a few well--known $q$--series results.  These results include the following:  

\begin{lemma}
\label{lemma1}
$$
\frac{f_2^2}{f_1} = \sum_{m\geq 0}q^{m(m+1)/2}
$$
\end{lemma}
\begin{proof}
This result appears as (1.5.3) in Hirschhorn \cite{H}.  
\end{proof}

\begin{lemma}
\label{lemma2}
$$
f_1^3 = \sum_{m\geq 0}(-1)^m(2m+1) q^{m(m+1)/2}.
$$
\end{lemma}
\begin{proof}
This result appears as (1.7.1) in Hirschhorn \cite{H}.  
\end{proof}

\begin{lemma}
\label{lemma8}
$$
f_1 = \sum_{m = -\infty}^{\infty}(-1)^m q^{m(3m-1)/2}.
$$
\end{lemma}
\begin{proof}
This result appears as (1.6.1) in Hirschhorn \cite{H}.  
\end{proof}

\begin{lemma}
\label{lemma3}
\begin{equation*}
    \frac{f_1^5}{f_2^2} = \sum_{m=-\infty}^{\infty} (6m+1)q^{m(3m+1)/2}.
\end{equation*}
\end{lemma}
\begin{proof}
See Berndt \cite{B}, Corollary 1.3.21 as well as Hirschhorn \cite{H}, (10.7.3).
\end{proof}

\begin{lemma}
\label{lemma4}
\begin{align}
    \frac{f_1^2}{f_2} & = \sum_{j=-\infty}^{\infty} (-1)^j q^{j^2}, \label{lem6.1} \\
    & = \frac{f_8^5}{f_4^2f_{16}^2} -2q\frac{f_{16}^2}{f_8}, \label{lem6.2} \\
    & = \frac{f_9^2}{f_{18}} -2q\frac{f_3f_{18}^2}{f_6f_9}. \label{lem6.3}
\end{align}
\end{lemma}
\begin{proof}
Identity \eqref{lem6.1} is equation (22.4) in Berndt \cite{B1}; see also Hirschhorn \cite{H}, (1.5.8). The 2-dissection \eqref{lem6.2} follows immediately from (1.9.4) in Hirschhorn \cite{H}. For a proof of \eqref{lem6.3} see \cite[(14.3.4)]{H}.
\end{proof}


\begin{lemma}
\label{lemma6}
\begin{equation*}
f_1f_2 = \frac{f_6f_9^4}{f_3f_{18}^2} -qf_9f_{18} -2q^2\frac{f_3f_{18}^4}{f_6f_9^2}.
\end{equation*}
\end{lemma}
\begin{proof}
See Hirschhorn and Sellers \cite{H-S}.
\end{proof}


\begin{lemma}
\label{lemma9}
\begin{equation*}
\frac{1}{f_1^4} = \frac{f_4^{14}}{f_2^{14}f_8^4} +4q\frac{f_4^2f_{8}^4}{f_2^{10}}.
\end{equation*}
\end{lemma}
\begin{proof}
This identity is (18) in Brietzke, da Silva, and Sellers \cite{BSS}.
\end{proof}

\begin{lemma}
\label{lemma_mod11_1}
$$
\frac{f_1^2 f_4^2}{f_2} = \sum_{m = -\infty}^{\infty} (3m+1)q^{3m^2+2m}.
$$
\end{lemma}
\begin{proof}
This result appears as (10.7.6) in Hirschhorn \cite{H}.  
\end{proof}

\begin{lemma}
\label{lemma_mod11_2}
$$
\frac{f_2^5}{f_1^2}= \sum_{m = -\infty}^{\infty} (-1)^m(3m+1)q^{3m^2+2m}.
$$
\end{lemma}
\begin{proof}
This result appears as (10.7.7) in Hirschhorn \cite{H}.  
\end{proof}

\section{Elementary Proofs of Several Congruences from Andrews and Paule}

In their work \cite{ap13}, Andrews and Paule used significant tools based on the work of Smoot, which are derived from modular forms, in order to prove a number of  congruences for the functions $d_2$ and $d_3.$  In this section, we wish to prove the majority of those results from Andrews and Paule using very elementary tools.  

\begin{theorem}[Corollary 5, \cite{ap13}]
\label{APCor5} 
For all $n\geq 0,$ $d_2(3n+2)\equiv 0\pmod{3}.$
\end{theorem}

\begin{proof}
We have the following:
\begin{eqnarray*}
\sum_{n=0}^{\infty} d_2(n)q^n 
&=& 
\frac{f_2^2}{f_1^{7}} \\
&=& 
\frac{f_2^2}{f_1}\frac{1}{f_1^6}\\
&\equiv& 
\frac{1}{f_3^2}
\left( 
\sum_{m\geq 0}q^{m(m+1)/2} \right)  \pmod{3}
\end{eqnarray*}
using Lemma \ref{lemma1}.  
Now we simply need to determine whether $3n+2=m(m+1)/2$ for some $m$ and $n.$  Completing the square means this is equivalent to determining whether 
$$8(3n+2)+1 = (2m+1)^2$$
or 
$$24n+17 = (2m+1)^2$$ 
or 
$$2 \equiv (2m+1)^2 \pmod{3}.$$  
This congruence never holds because 2 is a quadratic nonresidue modulo 3.  Our result follows.  
\end{proof}

\begin{theorem}[Corollary 10, \cite{ap13}]  
\label{APCor10}
For all $n\geq 0,$ $d_3(2n+1)\equiv 0\pmod{2}.$
\end{theorem}
\begin{proof}
Note that 
\begin{eqnarray*}
\sum_{n=0}^{\infty} d_3(n)q^n 
&=& 
\frac{f_2^3}{f_1^{10}} \\
&\equiv& 
\frac{f_2^3}{f_2^5} \pmod{2}\\
&\equiv& 
\frac{1}{f_2^2} \pmod{2}.
\end{eqnarray*}
Since $\frac{1}{f_2^2}$ is an even function of $q,$ the result follows.
\end{proof}

\begin{theorem}[Corollary 12, \cite{ap13}]  
\label{APCor12}
For all $n\geq 0,$ $d_3(4n+2)\equiv 0\pmod{2}.$
\end{theorem}
\begin{proof}
Thanks to the proof of Theorem \ref{APCor10}, we know 
\begin{eqnarray*}
\sum_{n=0}^{\infty} d_3(n)q^n 
&\equiv& 
\frac{1}{f_2^2} \pmod{2} \\
&\equiv &
\frac{1}{{f_4}} \pmod{2} 
\end{eqnarray*}
Since $\frac{1}{{f_4}}$ is a function of $q^4,$ the result follows.
\end{proof}	

\begin{theorem}[Corollary 13, \cite{ap13}]  
\label{APCor13}
For all $n\geq 0,$ $d_3(4n+3)\equiv 0\pmod{4}.$
\end{theorem}
\begin{proof}
Using the generating function for $d_3$, we have 
\begin{eqnarray*}
\sum_{n=0}^{\infty} d_3(n)q^n 
&=& 
\frac{f_2^3}{f_1^{10}} = \frac{f_2^4}{f_1^{12}} \frac{f_1^2}{f_2} \\
&\equiv& 
\frac{f_2^4}{f_2^6}\frac{f_1^2}{f_2} \pmod{4}\\
&=& 
\frac{1}{f_2^2}\left( \frac{f_8^5}{f_4^2f_{16}^2} -2q\frac{f_{16}^2}{f_8} \right)   
\end{eqnarray*}
using \eqref{lem6.2} in Lemma \ref{lemma4} above. Extracting the odd parts, dividing by $q$ and replacing $q^2$ by $q$, we are left with
$$\sum_{n=0}^{\infty} d_3(2n+1)q^n \equiv 2\frac{f_8^2}{f_1^2f_4} \equiv 2f_2^5 \pmod{4}.$$
Since $f_2$ is a function of $q^2$ the result follows.
\end{proof}

\begin{theorem}[Corollary 14, \cite{ap13}]  
\label{APCor14}
For all $n\geq 0,$ $d_3(5n+1)\equiv 0\pmod{5}.$
\end{theorem}
\begin{proof}
Using the generating function for $d_3$ which we have seen above, we have 
\begin{eqnarray*}
\sum_{n=0}^{\infty} d_3(n)q^n 
&=& 
\frac{f_2^3}{f_1^{10}} \\
&\equiv& 
\frac{f_2^3}{f_5^2} \pmod{5}\\
&=& 
\frac{1}{f_5^2}\left( \sum_{m=0}^\infty (-1)^m(2m+1)q^{m(m+1)} \right)   
\end{eqnarray*}
using Lemma \ref{lemma2} above. We now need to ask whether $5n+1$ can be represented as $m(m+1),$ and this is equivalent to asking whether $4(5n+1)+1$ or $20n+5$ can be represented as $(2m+1)^2.$  If this is the case, then we know 
$$
(2m+1)^2 = 20n+5 \equiv 0 \pmod{5}
$$
which implies that $2m+1 \equiv 0 \pmod{5}.$  Thanks to the presence of the coefficient of $2m+1$ in front of the term $q^{m(m+1)}$ in the series above, and the fact that this $2m+1$ must be divisible by 5, we know that our result follows.  
\end{proof}

\begin{theorem}[Corollary 15, \cite{ap13}]  
\label{APCor15}
For all $n\geq 0,$ 
\begin{eqnarray*}
d_3(5n+3) &\equiv& 0\pmod{5}, \\
d_3(5n+4) &\equiv& 0\pmod{5}. \\
\end{eqnarray*}
\end{theorem}
\begin{proof}
In the proof of Theorem \ref{APCor14}, we noted that 
\begin{eqnarray*}
\sum_{n=0}^{\infty} d_3(n)q^n 
&\equiv& 
\frac{1}{{f_5}^2}\left( \sum_{m=0}^\infty (-1)^m(2m+1)q^{m(m+1)} \right)  \pmod{5}. 
\end{eqnarray*}
We now need to ask whether $5n+3$ can be represented as $m(m+1),$ and this is equivalent to asking whether $4(5n+3)+1$ or $20n+13$ can be represented as $(2m+1)^2.$  This would mean that 
$(2m+1)^2 \equiv 3 \pmod{5}.$  
However, since 3 is a quadratic nonresidue modulo 5, we know that this cannot be the case.  Similarly, note that 
$$
4(5n+4)+1 = 20n+17 \equiv 2 \pmod{5}
$$
and 2 is the other quadratic nonresidue modulo 5.   The two congruences which appear in the statement of the theorem follow.  
\end{proof}

\section{New individual Congruences}

In this section, we state and prove a set of new individual congruences. We begin with an extremely unexpected congruence modulo 11 satisfied by the function $d_2$ (one of the functions which received a great deal of attention in Andrews and Paule \cite{ap13} and was the primary focus in Smoot's recent work \cite{Sm}).   

\begin{theorem} 
\label{theorem_d2_mod11}
For all $n\geq 0,$ $d_2(11n+7) \equiv 0 \pmod{11}.$
\end{theorem}

\begin{proof} 
We begin by noting that 
\begin{eqnarray*}
\sum_{n=0}^{\infty} d_2(n)q^n 
&=& 
\frac{f_2^2}{f_1^{7}} \\
&=& 
\frac{f_1^4f_2^2}{f_1^{11}}\\
&\equiv & 
\frac{f_1^4f_2^2}{f_{11}} \pmod{11}.
\end{eqnarray*} 
Replacing $q$ by $q^2$ yields 
\begin{eqnarray*}
\sum_{n=0}^{\infty} d_2(n)q^{2n} 
&\equiv & 
\frac{f_2^4f_4^2}{f_{22}} \pmod{11}.
\end{eqnarray*}
Hence, in order to prove our result, we simply need to focus on $f_2^4f_4^2$ mod 11.  
Thanks to Lemma \ref{lemma_mod11_1} and Lemma \ref{lemma_mod11_2}, we know 

\begin{eqnarray*}
f_2^4f_4^2
&= &
\frac{f_2^5}{f_1^2}\frac{f_1^2f_4^2}{f_2} \\
&= &
\sum_{j = -\infty}^{\infty} (-1)^j(3j+1)q^{3j^2+2j}
\sum_{k = -\infty}^{\infty} (3k+1)q^{3k^2+2k}. 
\end{eqnarray*}
Thus, at this point, we need to ask whether $2(11n+7)$ can be represented as 
$$2(11n+7) = 3j^2+2j+3k^2+2k.$$
This is equivalent to 
$$24(11n+7)+8 = (6j+2)^2+(6k+2)^2$$ 
which implies 
$$(6j+2)^2+(6k+2)^2 \equiv 0 \pmod{11}$$
and this implies 
$$(3j+1)^2+(3k+1)^2 \equiv 0 \pmod{11}.$$
Thus, if there is a representation of $2(11n+7)$ as $3j^2+2j+3k^2+2k,$ then we know that $11$ divides $3j+1$ and $11$ divides $3k+1.$  (Assume this is not true.  Then neither $3j+1$ nor $3k+1$ is divisible by 11. However, we know that the quadratic residues modulo 11 are 1, 3, 4, 5, and 9. So, $(3j+1)^2+(3k+1)^2 \not\equiv 0 \pmod{11},$ which is a contradiction.)  Thus, we now know that the coefficient of  $q^{3j^2+2j+3k^2+2k}$ in the series representation for $f_2^4f_4^2,$ which is $(-1)^j(3j+1)(3k+1),$ must be congruent to 0 modulo 11.  
\end{proof}

We next prove a theorem connecting $d_\ell(n)$ and $p(n)$, for each prime $\ell$, where $p(n)$ denotes the number of unrestricted partitions of $n.$ This theorem provides congruences for $d_5, d_7,$ and $d_{11}$ which are closely related to Ramanujan's congruences for $p(n).$

\begin{theorem}
\label{R1}
Let $\ell$ be a prime and let $r$, $1\leq r\leq \ell -1,$ be an integer such that $p(\ell n+r) \equiv 0 \pmod{\ell}$, for all $n \geq 0$. Then, for all $n \geq 0$,
\begin{equation*}
d_{\ell}(\ell n+r) \equiv 0 \pmod{\ell}.
\end{equation*}
\end{theorem}

\begin{proof}
The generating function for $d_{\ell}(n)$ satisfies 
\begin{align*}
 \sum_{n=0}^{\infty}d_{\ell}(n)q^{n} & = \frac{f_2^{\ell}}{f_1^{3\ell +1}} \equiv \frac{f_{2\ell}}{f_\ell^3}\frac{1}{f_1} \\
 & \equiv \frac{f_{2\ell}}{f_\ell^3} \sum_{n=0}^{\infty}p(n)q^n \pmod{\ell}.
\end{align*}
Since $f_{2\ell}/f_{\ell}^3$ is a function of $q^\ell$ and $p(\ell n+r) \equiv 0 \pmod{\ell}$, the result follows.
\end{proof}

\begin{corollary}
\label{Cn1}
For all $n\geq 0,$
\begin{eqnarray*}
d_5(5n+4) &\equiv & 0 \pmod{5}, \\
d_7(7n+5) &\equiv & 0 \pmod{7}, \\
d_{11}(11n+6) &\equiv & 0 \pmod{11}.
\end{eqnarray*}
\end{corollary}

Next, we focus our attention specifically on the function $d_7(n).$
\begin{theorem}  For all $n\geq 0,$
\begin{eqnarray}
d_7(4n+2) &\equiv& 0\pmod{4}, \label{8.4} \\
d_7(8n+5) &\equiv& 0\pmod{4}, \label{8.7}  \\
d_7(16n+9) &\equiv& 0\pmod{4}, \label{8.12} \\
d_7(4n+3) &\equiv& 0\pmod{8}, \label{8.5} \\
d_7(8n+4) &\equiv& 0\pmod{8}. \label{8.6}
\end{eqnarray}
\end{theorem}

\begin{proof}
Using the generating function for $d_7(n)$ and \eqref{lem6.2}, it follows that
\begin{align}
\sum_{n=0}^{\infty}d_7(n)q^n & = \frac{f_2^7}{f_1^{22}} = \frac{f_2^8}{f_1^{24}}\frac{f_1^2}{f_2} \label{8.15} \\
& \equiv \frac{1}{f_4^2}\left( \frac{f_8^5}{f_4^2f_{16}^2} -2q\frac{f_{16}^2}{f_8} \right) \pmod{4}. \label{8.8}
\end{align}
Extracting the even parts from \eqref{8.8} and replacing $q^2$ by $q$, we are left with
\begin{equation*}
\sum_{n=0}^{\infty}d_7(2n)q^n \equiv \frac{f_4^5}{f_2^4f_{8}^2} \equiv \frac{1}{f_4} \pmod{4},
\end{equation*}
from which \eqref{8.4} follows since $1/f_4$ is a function of $q^4$.

Extracting the odd parts of \eqref{8.8}, dividing by $q$ and replacing $q^2$ by $q$, we obtain
\begin{equation}
\sum_{n=0}^{\infty}d_7(2n+1)q^n \equiv 2\frac{f_8^2}{f_2^2f_4} \equiv 2f_8 \pmod{4}.    
\label{8.10}
\end{equation}

This means 
\begin{equation}
\sum_{n=0}^{\infty}d_7(4n+1)q^n \equiv 2{f_4} \pmod{4},
\label{8.13}
\end{equation}
from which \eqref{8.7} follows since $f_4$ is a function of $q^4.$ 

In a similar way, \eqref{8.13} implies
$$\sum_{n=0}^{\infty}d_7(8n+1)q^n \equiv 2{f_2} \pmod{4}.$$
Since $f_2$ is a function of $q^2$, \eqref{8.12} follows.

In order to prove \eqref{8.5} and \eqref{8.6}, we take \eqref{8.15} modulo $8$:
\begin{equation}
\sum_{n=0}^{\infty}d_7(n)q^n \equiv \frac{1}{f_2^4}\frac{f_1^2}{f_2} \pmod{8}.        
\label{8.16}
\end{equation}

Thanks to \eqref{lem6.2} in Lemma \ref{lemma4}, we see that the odd part of \eqref{8.16} is
\begin{eqnarray*}
\sum_{n=0}^{\infty}d_7(2n+1)q^n &\equiv & -2\frac{f_8^2}{f_1^4f_4} \pmod{8} \\
&=& -2\frac{f_8^2}{f_4}\left( \frac{f_4^{14}}{f_2^{14}f_8^4} +4q\frac{f_4^2f_{8}^4}{f_2^{10}} \right)  
\end{eqnarray*}
using Lemma \ref{lemma9}.  
It follows that
$$\sum_{n=0}^{\infty}d_7(4n+3)q^n \equiv 0 \pmod{8},$$
which proves \eqref{8.5}.

Using \eqref{lem6.2}, we can extract the even part of \eqref{8.16}:
$$\sum_{n=0}^{\infty}d_7(2n)q^n \equiv \frac{f_4^5}{f_1^4f_2^2f_8^2} \pmod{8}.$$
Now, using Lemma \ref{lemma9} we extract the even part of the last expression to obtain
$$\sum_{n=0}^{\infty}d_7(4n)q^n \equiv \frac{f_2^{19}}{f_1^{16}f_4^6} \equiv \frac{f_2^{11}}{f_4^6} \pmod{8},$$
from which \eqref{8.6} follows.
\end{proof}

\begin{theorem}
Let $p \geq 5$ be a prime and let $r$, $1 \leq r \leq p-1$, be such that $3r+1$ is a quadratic nonresidue modulo $p$.  Then, for all $n \geq 0$,
$$d_{7}(2pn+2r+1) \equiv 0 \pmod{4}.$$
\end{theorem}

\begin{proof}
Thanks to Lemma \ref{lemma8}, we can rewrite \eqref{8.10} as
\begin{equation}
\sum_{n=0}^{\infty}d_7(2n+1)q^n \equiv 2 \sum_{m = -\infty}^{\infty}q^{4m(3m-1)} \pmod{4}.    
\label{8.11}
\end{equation}
If $pn+r = 4m(3m-1)$, completing square we obtain $3pn+3r+1 = (6m-1)^2$, from which it follows that $3r+1 \equiv (6m-1)^2 \pmod{p}$. Since $3r+1$ is a quadratic nonresidue modulo $p$, there is no $m$ such that $pn+r = 4m(3m-1)$. Thus, the coefficient of $q^{pn+r}$ on the right-hand side of \eqref{8.11} is $0$ modulo $4$, which implies $d_{7}(2(pn+r)+1) \equiv 0 \pmod{4}.$
\end{proof}

We next consider a pair of congruences satisfied by the function $d_8(n).$

\begin{theorem} For all $n \geq 0$,
\begin{eqnarray}
d_8(3n+2) &\equiv& 0\pmod{9}, \label{8.1} \\
d_8(9n+3) &\equiv& 0\pmod{9}. \label{8.2} 
\end{eqnarray}
\end{theorem}

\begin{proof}
Using the generating function for $d_8(n)$ and \eqref{lem6.3} in Lemma \ref{lemma4}, we obtain
\begin{align}
\sum_{n=0}^{\infty}d_8(n)q^n & = \frac{f_2^8}{f_1^{25}} = \frac{f_2^9}{f_1^{27}}\frac{f_1^2}{f_2} \nonumber \\
& \equiv \frac{f_6^3}{f_3^9} \left( \frac{f_9^2}{f_{18}} -2q\frac{f_3f_{18}^2}{f_6f_9} \right)  \pmod{9} \nonumber \\
& = \frac{f_6^3f_9^2}{f_3^9f_{18}} -2q\frac{f_6^2f_{18}^2}{f_3^8f_9}, \label{8.3}
\end{align}
from which we immediately deduce \eqref{8.1}. 

We now extract the terms of the form $q^{3n}$ from both sides of \eqref{8.3} to obtain
$$\sum_{n=0}^{\infty}d_8(3n)q^{3n} \equiv \frac{f_6^3f_9^2}{f_3^9f_{18}} \pmod{9}.$$
Replacing $q^3$ by $q$, we are left with
\begin{align*}
\sum_{n=0}^{\infty}d_8(3n)q^{n} & \equiv \frac{f_2^3f_3^2}{f_1^9f_{6}} \equiv \frac{f_2^3}{f_3f_{6}}  \pmod{9} \\
& = \frac{1}{f_3f_6} \sum_{m\geq 0}(-1)^m(2m+1) q^{m(m+1)},
\end{align*}
using Lemma \ref{lemma2}. Since $1/f_3f_6$ is a function of $q^3$ and $m(m+1) \not\equiv 1 \pmod{3}$, we see that the coefficient of $q^{3n+1}$ on the right-hand side of the congruence above is 0 modulo $9$. Thus, \eqref{8.2} follows.
\end{proof}



\section{New Infinite Families of Congruences}
In this section, we provide multiple extensions of the results of Andrews and Paule above which, in turn, provide new infinite families of congruences satisfied by the functions $d_k(n).$ 

The next result is rather surprising, primarily because the moduli in question range across {\bf all} primes $p\geq 5.$  The proof of this theorem follows easily via elementary generating function manipulations along with a $q$--series identity (Lemma \ref{lemma3} above) of Ramanujan.


\begin{theorem}
\label{mainthm1.1}
Let $p \geq 5$ be a prime and let $r$, $1 \leq r \leq p-1$, be such that $24r+1$ is a quadratic nonresidue modulo $p$.  Then, for all $n \geq 0$,
$$d_{p-2}(pn+r) \equiv 0 \pmod{p}.$$
\end{theorem}

\begin{proof}
The generating function for $d_{p-2}(n)$ is 
\begin{align*}
\sum_{n=0}^{\infty} d_{p-2}(n)q^n & = \frac{f_2^{p-2}}{f_1^{3p-5}} \equiv \frac{f_1^5}{f_2^2}\frac{f_{2p}}{f_p^3} \pmod{p}.    
\end{align*}
Thanks to Lemma \ref{lemma3}, we have 
\begin{equation}
    \sum_{n=0}^{\infty} d_{p-2}(n)q^n \equiv \frac{f_{2p}}{f_p^3} \left( \sum_{m=-\infty}^{\infty} (6m+1)q^{m(3m+1)/2} \right) \pmod{p}.
\label{eq8}
\end{equation}
We want to know whether $m(3m+1)/2 = pn+r$, for some $m$ and $n$. This is equivalent to asking whether $24pn+24r+1 = (6m+1)^2$, which implies $24r + 1 \equiv (6m+1)^2 \pmod{p}$. However $24r+1$ is a quadratic nonresidue modulo $p$. Therefore $d_{p-2}(pn+r) \equiv 0 \pmod{p}.$
\end{proof}
The theorem above is a fairly standard and classic result, providing exactly $(p-1)/2$ congruences for each prime $p.$  Interestingly enough, numerical evidence indicates that there are actually $(p-1)/2 + 1$ or $(p+1)/2$ such congruences for each prime $p.$  We explain this additional ``special'' congruence for each prime $p\geq 5$ here.  
\begin{theorem}
\label{mainthm1.2}
Let $p \geq 5$ be a prime and let $t$, $1 \leq t \leq p-1$, be the unique value such that $24t+1 \equiv 0 \pmod{p}$.  Then, for all $n \geq 0$, 
$$d_{p-2}(pn+t) \equiv 0 \pmod{p}.$$  
\end{theorem} 

\begin{proof}
Thanks to Lemma \ref{lemma3} and \eqref{eq8} above, we need to ask whether $pn+t = m(3m+1)/2$ for some integers $m$ and $n$. Completing the square and looking the result modulo $p$ yields $(6m+1)^2 \equiv 24t + 1 \equiv 0 \pmod{p}$. So $p$ divides $(6m+1)^2$, which implies that $p$ divides $6m+1$.  Since the coefficient of $q^{m(3m+1)/2}$ in the series that appears in Lemma \ref{lemma3} is exactly $6m+1$, it follows that the coefficient in question is congruent to $0$ modulo $p$. 
\end{proof}

\begin{theorem}
\label{mainthm1.3}
Let $p \geq 5$ be a prime and let $r$, $1 \leq r \leq p-1$, be a quadratic nonresidue modulo $p.$  Then, for all $n \geq 0$, 
$$d_{p-1}(pn+r) \equiv 0 \pmod{p}.$$
\end{theorem} 
\begin{proof}
Using \eqref{gf}, we know that, for prime $p\geq 5,$ 
\begin{align*}
\sum_{n=0}^{\infty} d_{p-1}(n)q^n 
&= \frac{f_2^{p-1}}{f_1^{3p-2}} \\
&= \frac{f_2^{p}}{f_1^{3p}}\frac{f_1^{2}}{f_2} \\
&\equiv \frac{f_{2p}}{f_p^{3}}\frac{f_1^{2}}{f_2} \pmod{p} \\
&= \frac{f_{2p}}{f_p^{3}}\left(\sum_{j=-\infty}^{\infty} q^{j^2}  \right)  
\end{align*}
using Lemma \ref{lemma4} above.  The result immediately follows by recognizing that $\frac{f_{2p}}{f_p^{3}}$ is a function of $q^p$ and that $r$ has been defined to be a quadratic nonresidue modulo $p.$
\end{proof}

We next provide an overarching result which allows us to naturally generalize Theorems \ref{APCor5}--\ref{APCor15} above as well as Theorems \ref{mainthm1.1},  \ref{mainthm1.2}, and \ref{mainthm1.3}.   

\begin{theorem}
\label{mainthm2}
Let $p$ be a prime, $k \geq 1,$ $j \geq 0,$ and $r$ be an integer such that $1\leq r \leq p-1.$  If, for all $n\geq 0,$ $$d_k(pn+r) \equiv 0 \pmod{p},$$ then for all $n\geq 0,$
$$d_{pj+k}(pn+r) \equiv 0 \pmod{p}.$$
\end{theorem}

\begin{proof}
Replacing $k$ by $pj+k$ in \eqref{gf} we obtain
\begin{align*}
\sum_{n=0}^{\infty} d_{pj+k}(n)q^n & = \frac{f_2^{pj+k}}{f_1^{3pj+3k+1}} = \frac{f_2^k}{f_1^{3k+1}}\frac{f_2^{pj}}{f_1^{3pj}} \\
& \equiv \frac{f_{2p}^{j}}{f_p^{3j}}\sum_{m=0}^{\infty} d_{k}(m)q^m \pmod{p}.
\end{align*}
All the exponents of $q$ coming from ${f_{2p}^{j}}/{f_p^{3j}}$ are of the form of $pN$. Since $d_{k}(pm+r) \equiv 0 \pmod{p}$, it follows that $d_{pj+k}(pn+r) \equiv 0 \pmod{p}$.
\end{proof}

Theorem \ref{mainthm2} provides a tool for writing down infinitely many new congruences with ease.  We exhibit such a list of new congruences below, using a shorthand notation to consolidate the statement of the results.  In what follows, the notation 
$$
An+B_1, B_2, \dots, B_t
$$
means we are considering the set of arithmetic progressions 
$$
An+B_1, An+B_2, \dots, An+B_t.
$$

\begin{corollary}
\label{cor_list}
For all $j \geq 0$ and $n \geq 0$,
\begin{align}
{d_{2j+1}(2n+1)} & \equiv 0 \pmod{2}, \label{eq0.1} \\
d_{3j+2}(3n+2) & \equiv 0 \pmod{3}, \label{eq0} \\
d_{5j + 3}(5n+1,3,4) & \equiv 0 \pmod{5}, \label{eq2} \\
d_{5j + 4}(5n+2,3) & \equiv 0 \pmod{5}, \label{eq4.1} \\
d_{5j + 5}(5n+4) &  \equiv 0 \pmod{5}, \label{n1.1} \\
d_{7j+5}(7n+2,3,4,6) & \equiv 0 \pmod{7},\label{eq7} \\
d_{7j+6}(7n+3,5,6) & \equiv 0 \pmod{7},\label{eq7.4} \\
 d_{7j + 7}(7n+5) &  \equiv 0 \pmod{7}, \label{n1.2} \\
d_{11j + 2}(11n+7) &  \equiv 0 \pmod{11}, \label{new11mod2} \\
d_{11j+9}(11n+3,5,6,8,9,10) & \equiv 0 \pmod{11},\label{eq11.1} \\
d_{11j+10}(11n+2,6,7,8,10) & \equiv 0 \pmod{11},\label{eq11.2} \\
d_{11j + 11}(11n+6) &  \equiv 0 \pmod{11}, \label{n1.3} \\
d_{13j+11}(13n+3,4,6,7,8,10,11) & \equiv 0 \pmod{13}, \label{eq13.1} \\
d_{13j+12}(13n+2,5,6,7,8,11) & \equiv 0 \pmod{13}. \label{eq13.2}
\end{align}
\end{corollary}
\begin{remark}
The corollary above does not provide an exhaustive list of congruences satisfied by these functions.  Our goal in writing these here is to provide the reader with a representative set of the kinds of congruences that arise within this family of partition functions.  
\end{remark}
\begin{proof}(of Corollary \ref{cor_list})
Thanks to Theorem \ref{mainthm2}, we only have to check the basis case, $j=0$, for all of the above results. To prove \eqref{eq0.1}, we note that taking $k=1$ in \eqref{gf} we are left with
$$\sum_{n=0}^{\infty} d_{1}(m)q^m  =  \frac{f_2}{f_1^{4}} \equiv \frac{1}{f_2} \pmod{2},$$
from which we see that $d_1(2n+1) \equiv 0 \pmod{2}$ and then \eqref{eq0.1} follows.


The basis case for \eqref{eq0} is proven in Theorem \ref{APCor5}.


The basis cases for the congruences listed in \eqref{eq2} are proven in Theorems \ref{APCor14} and \ref{APCor15}.

The basis cases for \eqref{eq4.1} and \eqref{eq7.4} follow from the $p=5$ and $p=7$ cases of Theorem \ref{mainthm1.3} above, respectively.

The basis cases for \eqref{n1.1}, \eqref{n1.2}, and \eqref{n1.3} follow from Corollary \ref{Cn1} above.

The basis cases for \eqref{eq7}--\eqref{eq13.2} follow immediately from Theorems \ref{mainthm1.1}, \ref{mainthm1.2}, and \ref{mainthm1.3}.

The basis case for \eqref{new11mod2} follows from Theorem \ref{theorem_d2_mod11} above.

\end{proof}

We close this section with two infinite families of congruences modulo 9 which are closely related to congruences that appear in Andrews and Paule \cite{ap13}.

\begin{theorem}
For all $j \geq 0$ and $n \geq 0$,
\begin{eqnarray*}
d_{9j+2}(9n+5) &\equiv& 0\pmod{9}, \\ 
d_{9j+2}(9n+8) &\equiv& 0\pmod{9}.    
\end{eqnarray*}
\end{theorem}

\begin{proof}
Using the generating function for $d_2$, we have 
\begin{eqnarray*}
\sum_{n=0}^{\infty} d_2(n)q^n 
&=& 
\frac{f_2^2}{f_1^7} = \frac{1}{f_1^9} f_1^2f_2^2 \\
&\equiv& 
\frac{1}{f_3^3} f_1^2f_2^2 \pmod{9}\\
&\equiv& 
\frac{1}{f_3^3}\left( \frac{f_6f_9^4}{f_3f_{18}^2} -qf_9f_{18} -2q^2\frac{f_3f_{18}^4}{f_6f_9^2} \right)^2   \pmod{9}
\end{eqnarray*}
using Lemma \ref{lemma6} above. Extracting the terms involving $q^{3n+2}$, dividing by $q^2$ and replacing $q^3$ by $q$, we obtain
$$\sum_{n=0}^{\infty} d_2(3n+2)q^n \equiv -3f_3^2f_6^2 \pmod{9}.$$
Since $f_3$ and $f_6$ are functions of $q^3$ it follows that
\begin{eqnarray*}
d_{2}(9n+5) &\equiv& 0\pmod{9}, \\
d_{2}(9n+8) &\equiv& 0\pmod{9}, 
\end{eqnarray*}
which provides the basis case for an inductive proof. Thus, let us assume that $d_{9j+2}(9n+5) \equiv 0\pmod{9}$ and $d_{9j+2}(9n+8) \equiv 0\pmod{9}$ for some $j$. In order to complete the proof, we note that 
\begin{align}
\sum_{n=0}^{\infty}d_{9(j+1)+2}(n)q^{n} & = \frac{f_2^{9(j+1)+2}}{f_1^{27(j+1)+7}} = \frac{f_2^9}{f_1^{27}}\frac{f_2^{9j+2}}{f_1^{27j+7}} \nonumber \\
& = \frac{f_2^9}{f_1^{27}} \sum_{m=0}^{\infty}d_{9j+2}(m)q^{m} \nonumber \\
& \equiv \frac{f_6^3}{f_9^{3}} \sum_{m=0}^{\infty}d_{9j+2}(m)q^{m}  \pmod{9}. \label{6.12.1}
\end{align}
Using Lemma \ref{lemma2}, we have
$$f_6^3 = \sum_{k \geq 0}(-1)^k(2k+1)q^{3k(k+1)}.$$
For all $k \geq 0$, we know that $3k(k+1) \equiv 0,6 \pmod{9}$, with $3k(k+1) \equiv 6 \pmod{9}$ if and only if $k \equiv 1,4,7 \pmod{9}$, which is equivalent to $2k+1 \equiv 3,9,15 \pmod{9}$. Thus, modulo $9$, $f_6^3$ becomes
$$f_6^3 \equiv \sum_{k \geq 0}a_kq^{9k} + 3 \sum_{k \geq 0}b_kq^{9k+6} \pmod{9},$$
for certain integer coefficients $a_k$ and $b_k$. Hence, using \eqref{6.12.1} we obtain
\begin{align*}
\sum_{n=0}^{\infty}d_{9(j+1)+2}(n)q^{n} & \equiv \frac{1}{f_9^{3}} \left(  \sum_{k \geq 0}a_kq^{9k}\sum_{m=0}^{\infty}d_{9j+2}(m)q^{m} \right. \\
& \ \ \ \ \left. + 3 \sum_{k \geq 0}b_kq^{9k+6}\sum_{m=0}^{\infty}d_{9j+2}(m)q^{m} \right) \pmod{9}.
\end{align*}
By the hypothesis, we see that the coefficients of $q^{9n+5}$ and $q^{9n+8}$ in the first term of the congruence above are congruent to $0$ modulo $9$. For the second term in the congruence above, namely
\begin{equation}
\sum_{m,k \geq 0} 3b_kd_{9j+2}(m)q^{9k+m+6},
\label{r1}
\end{equation}
we note that $9k+m+6 \equiv 5 \pmod{9}$ if and only if $m \equiv 8 \pmod{9}$. In this case, we know that $m = 9N+8$ for some $N$. By \eqref{eq0} in Corollary \ref{cor_list}, we have $d_{9j+2}(9n+8) \equiv 0 \pmod{3}$, for all $n$, which yields $3b_kd_{9j+2}(9N+8) \equiv 0 \pmod{9}$. Thus the coefficients of $q^{9n+5}$ in \eqref{r1} are congruent to $0$ modulo $9$. Analogously, we see that $9k+m+6 \equiv 8 \pmod{9}$ if and only if $m \equiv 2 \pmod{9}$. So, $m = 9N+2$ for some $N$. However, we also know from \eqref{eq0} in Corollary \ref{cor_list} that $d_{9j+2}(9n+2) \equiv 0 \pmod{3}$, for all $n$. Thus, $3b_kd_{9j+2}(9N+2) \equiv 0 \pmod{9}$, from which we conclude that the coefficients of $q^{9n+8}$ in \eqref{r1} are congruent to $0$ modulo $9$.
\end{proof}

\section{Concluding Remarks}
We readily admit that the above results are not an exhaustive list of all of the congruences satisfied by the functions $d_k(n).$  Indeed, there are many other congruences to consider.  For example, computational evidence hints at the possibility of infinite families of congruences modulo arbitrarily high powers of 3 for the function $d_{11}(n)$ as well as the possibility of infinite families of congruences modulo powers of 2 for $d_7(n)$ and $d_{15}(n)$ (in the spirit of the work completed by Smoot for $d_2$ modulo powers of 3).  Indeed, for $k=2^j-1$ for some $j,$ it is clear that the generating function for $d_k(n)$ will have a structure that allows for a number of congruences to hold for small moduli.  One must wonder whether an {\bf infinite} family of congruences, modulo powers of 2, like the family in Theorem \ref{SmootCongs} holds for these special values of $k.$  We leave such an investigation to the interested reader.


\section*{Acknowledgments}

The first author was supported by S\~ao Paulo Research Foundation (FAPESP) (grant no. 2019/14796-8).

\end{document}